\documentclass[11pt,a4paper]{amsart}
\usepackage[T1]{fontenc}
\usepackage[utf8]{inputenc}
\usepackage{lmodern}
\usepackage{amsmath,amssymb,mathtools}
\usepackage{microtype}
\usepackage[margin=29mm]{geometry}
\usepackage{xcolor}
\usepackage[colorlinks=true,linkcolor=blue!45!black,citecolor=blue!45!black,urlcolor=blue!45!black]{hyperref}
\usepackage{enumitem}
\usepackage{needspace}
\hypersetup{
 pdftitle={Hardy inequalities on the positive integer lattice: a dimensional transition},
 pdfauthor={Carlos Lizama},
 pdfsubject={Sharp Hardy constants and a dimensional transition on the positive integer lattice},
 pdfkeywords={Hardy inequality, positive integer lattice, sharp constant, essential spectrum}
}
\numberwithin{equation}{section}
\newtheorem{theorem}{Theorem}[section]
\newtheorem{proposition}[theorem]{Proposition}
\newtheorem{lemma}[theorem]{Lemma}
\newtheorem{corollary}[theorem]{Corollary}
\theoremstyle{remark}

\newcommand{\N}{\mathbb N}
\newcommand{\R}{\mathbb R}

\newcommand{\Om}{\Omega_N}
\newcommand{\dd}{\,\mathrm d}
\newcommand{\Cp}{C_+(N)}
\newcommand{\cfin}{c_{00}}
\newcommand{\ess}{\mathrm{ess}}
\newcommand{\qform}{\mathfrak q_N}
\newcommand{\wform}{\mathfrak w_N}
\DeclareMathOperator{\supp}{supp}
\DeclareMathOperator{\Dom}{Dom}

\allowdisplaybreaks[1]
\title[Hardy inequalities on the positive lattice]{Hardy inequalities on the positive integer lattice:\\ a dimensional transition}
\author{Carlos Lizama}
\address{Departamento de Matem\'atica y Ciencia de la Computaci\'on, Universidad de Santiago de Chile, Santiago, Chile}
\email{carlos.lizama@usach.cl}
\urladdr{https://orcid.org/0000-0002-9807-1100}
\date{}
\thanks{This work was supported by ANID--FONDECYT Regular project 1260057.}

\begin{document}

\begin{abstract}
We determine exactly in which dimensions the sharp Hardy coefficient on the positive integer lattice agrees with that on the positive Euclidean cone, with Dirichlet conditions on every coordinate face. Equality holds in dimensions one and two, whereas the discrete coefficient is strictly smaller in every dimension at least three. An exact transformation identifies the lattice problem with a weighted Dirichlet problem on a cube. A tangential integration identity evaluates the defect of an explicit critical profile and proves the dimensional transition. The continuous coefficient remains the lower edge of the essential spectrum of the transformed operator; below that edge, the sharp discrete coefficient is a simple eigenvalue with a positive eigenfunction. In dimension three, an explicit supersolution places the coefficient strictly between twelve and forty-nine quarters. A representation on a product of three-dimensional spheres relates the spectral threshold to the local geometry of the transformed problem.
\end{abstract}

\maketitle

\section{Introduction}

Replacing derivatives by finite differences is a basic way to pass from continuous models to lattice models. Finite differences recover derivatives on scales large compared with the lattice spacing, but this agreement need not preserve the best constants in an inequality. Hardy inequalities make this distinction precise: they compare the energy required for a function to vary with its size measured by an inverse-square weight. Their best coefficient is also the threshold beyond which an attractive inverse-square potential produces negative energy. Determining this coefficient therefore connects a question about inequalities with a question about stability. The classical theory and its geometric and spectral consequences are discussed in \cite{HLP,Davies}.

For the inverse-square weight on nearest-neighbour lattices, the exact range of dimensions in which the continuous coefficient survives has only recently been resolved. On the full integer lattice, the continuous coefficient is preserved through dimension eight and fails from dimension nine onwards \cite{AlpayEarly,LizamaEarly,Alpay,LizamaFull}. This exact division into dimensions with equality and dimensions with strict inequality identifies a phenomenon that asymptotic agreement between differences and derivatives does not explain. It raises a broader question: which features of a discrete energy, a prescribed weight and a boundary condition determine whether the best continuous constant survives discretization?

We study this comparison when every coordinate is positive and functions vanish on the coordinate faces. In one dimension the continuous region is a half-line; in two dimensions it is the first quadrant. In arbitrary dimension we call it the positive cone. Its discrete counterpart, the positive integer lattice, carries the nearest-neighbour difference energy. The boundary conditions are essential: they change the sharp continuous coefficient and may also change the dimensions in which discretization preserves it.

We prove that the continuous and discrete coefficients agree exactly in dimensions one and two. From dimension three onwards, the discrete coefficient is strictly smaller. The proof also explains why this happens. An exact transformation converts the lattice problem into a continuous problem on a cube. Concentration near one corner recovers the continuous coefficient, while an identity over the whole cube detects the strict decrease. The transition thus separates the information supplied by local geometry from the information needed to determine the global optimum. The resulting transition from dimension two to dimension three shows that the critical dimension changes when Dirichlet faces are imposed. Together with the full-lattice results and recent weighted one-dimensional comparisons \cite{Barki}, this gives a concrete starting point for a systematic study of preservation and loss of sharp Hardy constants under discretization.

\subsection{The problem and its dimensional classification}

Let $\N=\{1,2,\ldots\}$. We work on the positive integer lattice
\[
 \N^N=\{n=(n_1,\ldots,n_N): n_j\geq 1\},
 \qquad N\geq1.
\]
Every function on this set is extended by zero to the coordinate faces. The Dirichlet lattice Laplacian is
\begin{equation}\label{eq:laplacian}
 (L_+u)(n)=2Nu(n)-\sum_{j=1}^N\bigl(u(n+e_j)+u(n-e_j)\bigr),
\end{equation}
where $e_j$ is the $j$th coordinate vector. Write $|n|^2=\sum_j n_j^2$, and let $\cfin(\N^N)$ denote the finitely supported functions. For such a function $u$, define
\begin{equation}\label{eq:discrete-forms}
 \mathcal E_+(u)=\sum_{n\in\N^N}\sum_{j=1}^N|u(n)-u(n-e_j)|^2,
 \qquad
 \mathcal D_+(u)=\sum_{n\in\N^N}\frac{|u(n)|^2}{|n|^2}.
\end{equation}
Thus $\mathcal E_+(u)=\langle L_+u,u\rangle$. Each interior edge is counted once, and edges meeting a Dirichlet face contribute to the energy. The largest coefficient for which $\mathcal E_+(u)\geq C_+(N)\mathcal D_+(u)$ holds for every such $u$ is
\begin{equation}\label{eq:C-def}
 \Cp=\inf_{0\ne u\in\cfin(\N^N)}
                  \frac{\mathcal E_+(u)}{\mathcal D_+(u)}.
\end{equation}

The continuous benchmark is the Hardy inequality on the positive cone $\mathcal O_N=(0,\infty)^N$:
\[
 \int_{\mathcal O_N}|\nabla f|^2\dd x
 \geq A_N\int_{\mathcal O_N}\frac{|f|^2}{|x|^2}\dd x,
 \qquad f\in C_c^\infty(\mathcal O_N).
\]
Its sharp coefficient is
\begin{equation}\label{eq:A-def}
 A_N=\frac{(3N-2)^2}{4}.
\end{equation}
This value follows by separating radial variation from the angular contribution of the Dirichlet boundary. Su and Yang \cite{SuYang} obtain the more general coefficient when only a prescribed number of coordinates are required to be positive. We recall the calculation relevant here in Section~\ref{sec:duality}.

Our main result determines precisely when the positive lattice preserves this benchmark. The one-dimensional value recovers the classical discrete Hardy inequality.

\begin{theorem}\label{thm:main}
The constants in \eqref{eq:C-def} satisfy
\begin{equation}\label{eq:classification}
 C_+(1)=\frac14,\qquad C_+(2)=4,\qquad
 \Cp<A_N\quad\text{for every }N\geq3.
\end{equation}
In dimension three,
\begin{equation}\label{eq:three-bounds}
                  12<C_+(3)<\frac{49}{4}.
\end{equation}
In particular, $\Cp=A_N$ holds exactly in dimensions one and two.
\end{theorem}

The theorem gives a complete classification of equality with the continuous benchmark. Its new higher-dimensional content consists of the exact value in dimension two and the strict decrease in every dimension at least three. The bounds in \eqref{eq:three-bounds} locate the first case of this decrease; the exact value of $C_+(3)$ remains open. The proof of the classification is uniform in dimension and depends on the global identity \eqref{eq:total-defect}.

\subsection{The mechanism and its spectral meaning}

To compare the two coefficients, we use an exact continuous formulation. Set
\begin{equation}\label{eq:cube-symbol}
 \Om=(0,\pi)^N,\qquad
 a(\theta)=4\sum_{j=1}^N\sin^2\frac{\theta_j}{2}.
\end{equation}
On $H_0^1(\Om)$, write
\begin{equation}\label{eq:continuous-forms}
 \qform[v]=\int_{\Om}|\nabla v|^2\dd\theta,
 \qquad
 \wform[v]=\int_{\Om}\frac{|v|^2}{a}\dd\theta.
\end{equation}
The sine transform and a duality argument show that $\Cp$ is the infimum of $\qform/\wform$. This equality preserves the optimal coefficient exactly. Near the corner $\theta=0$, the function $a(\theta)$ agrees with $|\theta|^2$ to leading order, so the weight has the same inverse-square singularity as the continuous benchmark. That local agreement fixes the threshold for functions concentrating at the corner. The behaviour over the whole cube determines whether the infimum lies below it.

The next statement expresses this distinction spectrally. Here $\sigma$ denotes the spectrum and $\sigma_{\ess}$ the essential spectrum. The operator $K_N$ is defined through its closed quadratic form, so its domain includes the Dirichlet conditions. The multiplier $M_{|n|^{-2}}$ acts by $u(n)\mapsto |n|^{-2}u(n)$.

\begin{theorem}\label{thm:spectral-main}
Let $\mathcal H_N=L^2(\Om,a^{-1}\dd\theta)$, and let $K_N$ be the self-adjoint operator associated with the form $\qform$ on $H_0^1(\Om)$. Then
\begin{equation}\label{eq:dual-and-essential}
 \Cp=\inf_{0\ne v\in H_0^1(\Om)}\frac{\qform[v]}{\wform[v]}
       =\inf\sigma(K_N),
 \qquad \sigma_{\ess}(K_N)=[A_N,\infty).
\end{equation}
The infimum is not attained for $N=1,2$. For every $N\geq3$, it is an isolated simple eigenvalue. Its normalized eigenfunction can be chosen strictly positive in $\Om$ and invariant under coordinate permutations.

For $c\in\R$, the discrete operator
\begin{equation}\label{eq:original-schrodinger}
 H_N(c)=L_+-cM_{|n|^{-2}}
\end{equation}
has essential spectrum $[0,4N]$. It has a negative eigenvalue if and only if $c>\Cp$.
\end{theorem}

For $N\geq3$, the decrease therefore produces a positive minimizing function in the transformed problem. In the original lattice problem, the same coefficient marks the onset of a negative eigenvalue as the attraction strength increases. These are two consequences of the same variational threshold.

The proof of the decrease rests on an explicit critical profile: a positive function with the corner behaviour of the continuous Hardy problem. Its defect in the transformed differential equation is nonnegative in dimensions one and two. In higher dimensions, derivatives tangent to the level sets of $a$ give an exact identity showing that the integrated defect is negative. Cutting off the corner singularity at a vanishing energy cost turns this sign into a test function with quotient below $A_N$. A further representation on a product of three-dimensional spheres explains the value of the threshold through the local dimension $3N$.

\subsection{Related work, contribution and perspective}

The equality of discrete and continuous Hardy constants has a classical one-dimensional precedent \cite{HLP}. Recent weighted and higher-order results substantially extend that comparison. Gupta \cite{GuptaWeighted} proves sharp weighted inequalities on the half-line for ranges of power exponents. Huang and Ye \cite{HuangYe} obtain sharp discrete Hardy--Rellich inequalities of arbitrary integer order and a discrete Hardy--Leray inequality with the continuous coefficient. Barki \cite{Barki} studies power-weighted discrete $p$-energies: the sharp continuous coefficient is retained for nonnegative weight exponents outside the critical exponent, whereas equality fails for certain negative exponents. These results identify the weight as one parameter governing preservation of the sharp constant.

For multidimensional lattices, Kapitanski and Laptev \cite{KapitanskiLaptev} develop continuous and discrete Hardy inequalities. Gupta \cite{Gupta} proves that the best inverse-square coefficient on the full lattice grows linearly with dimension, whereas the continuous coefficient grows quadratically. Thus failure of equality in sufficiently large dimensions was already known. Recent results of Gupta \cite{GuptaAsymptotic} and Huang and Ye \cite{HuangYeAsymptotic} identify the leading asymptotic coefficient, also for higher-order inequalities. Dyda \cite{Dyda} treats fractional inequalities and characterizes when power-weighted non-fractional inequalities hold, with explicit constants that are not asserted to be optimal.

The exact dimensional transition for the Euclidean inverse-square weight has been established in recent preprints. Alpay \cite{AlpayEarly} proves the sharp result in dimension three. Lizama \cite{LizamaEarly} proves equality in dimensions three and four and strict inequality from dimension nine, and Alpay \cite{Alpay} settles equality in dimensions five through eight. Lizama \cite{LizamaFull} gives a complete dimensional treatment and relates the transition to the total critical defect and the onset of spectral instability. The recent advance is therefore the exact equality range and the mechanism of the transition. The present paper establishes a second dimensional classification, with Dirichlet conditions on all coordinate faces, and identifies its transition by a different global identity adapted to the boundary.

The central new identity is obtained by differentiating tangentially to the level sets of the lattice symbol after a sine transformation. It reduces the integrated critical residual to an explicit dimension-dependent multiple of a positive mass and a nonpositive correction. Its sign gives strict inequality simultaneously in every dimension at least three. The exact continuous formulation, the identification of its essential spectrum, and the product-of-spheres representation explain how the continuous coefficient remains a local threshold while the discrete optimum falls below it. This provides both a classification and an analytic mechanism for the transition in the presence of Dirichlet faces.

A complementary theory optimizes the Hardy weight itself. Keller, Pinchover and Pogorzelski \cite{KPP} construct optimal weights for graph Schr\"odinger operators, and Keller and Lemm \cite{KellerLemm} study their inverse-square asymptotics on Euclidean lattices. On the positive lattice, \cite[Example~4.4]{KPPRellich} gives a weight of the form $\sum_j w(n_j)$, where $w(m)>1/(4m^2)$. This controls a sum of coordinate singularities. Here the prescribed weight is the radial function $1/(\sum_jn_j^2)$, and the scalar multiplying it is optimized. The distinction matters because an optimal graph weight can share a continuous asymptotic coefficient without determining the best scalar for a fixed radial weight.

The continuous literature supplies many further comparison problems. Su and Yang \cite{SuYang} compute the sharp inverse-square coefficient when any specified number of coordinate directions is positive. Ghoussoub and Moradifam \cite{GhoussoubMoradifam} characterize pairs of radial weights through a second-order ordinary differential equation, providing a framework for optimal weighted Hardy and Hardy--Rellich inequalities. Gupta \cite{GuptaAntisymmetric} studies the effect of antisymmetry in Euclidean space, on a torus and on the integer lattice. Antisymmetry imposes vanishing on coordinate-collision sets and changes the admissible class; in the present problem the positive minimizing function is invariant under coordinate permutations. These examples suggest varying weights, geometry and symmetry as separate parameters in the discrete--continuous comparison.

The perspective is therefore to determine, for each specified discretization, the parameter range in which the continuous optimum is preserved and the mechanism that produces a strict decrease. The current theorem supplies a complete answer when the parameter is dimension and all coordinate faces carry Dirichlet conditions. Section~\ref{sec:questions} formulates further comparison questions, while keeping the explicit evaluation of $C_+(3)$ as a separate unresolved problem.

Section~\ref{sec:duality} establishes the exact variational equivalence. Section~\ref{sec:transition} proves the dimensional transition by evaluating the integrated critical residual. Section~\ref{sec:spectrum} identifies the essential spectrum and the discrete stability threshold. Section~\ref{sec:three} proves the strict lower bound in dimension three. Section~\ref{sec:geometry} gives the representation on a product of spheres and explains the effective dimension $3N$.

\section{An exact continuous formulation}\label{sec:duality}

The lattice Laplacian becomes multiplication by $a$ under the sine transform, while multiplication by $|n|^{-2}$ becomes the inverse Dirichlet Laplacian. Exchanging these two positive operators through their quadratic forms gives the continuous quotient. We first control its singular weight so that all operator products used in the argument are well defined.

\subsection{The local Hardy inequality}

Let $\mathcal O_N=(0,\infty)^N$. For $N\geq2$, set
\[
 \Sigma_N=\{\omega\in\mathbb S^{N-1}:\omega_j>0\text{ for all }j\},
 \qquad Y_N(\omega)=\prod_{j=1}^N\omega_j.
\]
The positive function $Y_N$ vanishes on $\partial\Sigma_N$. Since $x_1\cdots x_N$ is harmonic and homogeneous of degree $N$,
\begin{equation}\label{eq:angular-eigenvalue}
 -\Delta_{\mathbb S^{N-1}}Y_N=\mu_NY_N,
 \qquad \mu_N=2N(N-1).
\end{equation}
The ground-state identity on $\Sigma_N$ shows that $\mu_N$ is its first Dirichlet eigenvalue. For $N=1$, we take $\Sigma_1$ to be a single point, $Y_1=1$, and $\mu_1=0$.

This angular calculation yields the continuous estimate needed throughout the paper.

\begin{lemma}\label{lem:local-hardy}
For every $v\in H_0^1(\Om)$,
\begin{equation}\label{eq:local-hardy}
 \int_{\Om}|\nabla v|^2\dd\theta
 \geq A_N\int_{\Om}\frac{|v|^2}{|\theta|^2}\dd\theta.
\end{equation}
Consequently,
\begin{equation}\label{eq:weight-control}
 \wform[v]\leq\frac{\pi^2}{4A_N}\qform[v].
\end{equation}
\end{lemma}

\begin{proof}
We first explain the angular contribution. For $N\geq2$, the polar-coordinate formula for the Laplacian gives
\[
 \Delta(r^NY_N(\omega))
 =r^{N-2}\bigl[N(2N-2)Y_N+\Delta_{\Sigma_N}Y_N\bigr].
\]
The left-hand side is zero because $r^NY_N=x_1\cdots x_N$ is harmonic. This proves \eqref{eq:angular-eigenvalue}. To identify the first eigenvalue, let $\psi\in C_c^\infty(\Sigma_N)$ and write $\psi=Y_N f$. Expanding the square and integrating the cross term yields
\[
 \int_{\Sigma_N}|\nabla_\omega\psi|^2\dd\omega
 -\mu_N\int_{\Sigma_N}|\psi|^2\dd\omega
 =\int_{\Sigma_N}Y_N^2|\nabla_\omega f|^2\dd\omega\geq0.
\]
The inequality extends to $H_0^1(\Sigma_N)$ by density. Since $Y_N$ itself belongs to that space and solves the eigenvalue equation, the first eigenvalue is exactly $\mu_N$.

Now let $v\in C_c^\infty(\mathcal O_N)$, and put $b=(N-2)/2$. For $g\in C_c^\infty(0,\infty)$, expansion gives
\[
 \int_0^\infty |(r^bg)'|^2r\dd r
 =\int_0^\infty |g'|^2r^{N-1}\dd r
 +b^2\int_0^\infty |g|^2r^{N-3}\dd r
 +b\int_0^\infty (|g|^2)'r^{N-2}\dd r.
\]
The last integral is $-2b\int_0^\infty |g|^2r^{N-3}\dd r$, so
\begin{equation}\label{eq:radial-hardy-identity}
 \int_0^\infty |g'|^2r^{N-1}\dd r
 -b^2\int_0^\infty |g|^2r^{N-3}\dd r
 =\int_0^\infty |(r^bg)'|^2r\dd r\geq0.
\end{equation}
Apply this identity at each angular point, and apply the angular inequality at each radius. The polar decomposition of the energy then gives
\begin{align*}
 \int_{\mathcal O_N}|\nabla v|^2\dd x
 &=\int_0^\infty\int_{\Sigma_N}
 \bigl(|\partial_rv|^2+r^{-2}|\nabla_\omega v|^2\bigr)r^{N-1}\dd\omega\dd r\\
 &\geq (b^2+\mu_N)\int_{\mathcal O_N}\frac{|v|^2}{r^2}\dd x.
\end{align*}
Here $b^2+\mu_N=(3N-2)^2/4=A_N$. When $N=1$, there is no angular integral and \eqref{eq:radial-hardy-identity} gives the same conclusion with $b=-1/2$ and $\mu_1=0$.

Extend $v\in C_c^\infty(\Om)$ by zero to $\mathcal O_N$. This proves \eqref{eq:local-hardy} on that dense subspace. If $v_m\to v$ in $H_0^1(\Om)$, applying the inequality to $v_m-v_l$ shows that $v_m/r$ is Cauchy in $L^2$. Its limit is $v/r$, as follows by taking an almost-everywhere convergent subsequence. Thus the inequality also holds for $v\in H_0^1(\Om)$.

Finally, concavity of the sine function on $[0,\pi/2]$ gives $\sin(t/2)\geq t/\pi$ for $0\leq t\leq\pi$, while $\sin(t/2)\leq t/2$. Summing their squares gives
\begin{equation}\label{eq:a-comparison}
 \frac4{\pi^2}|\theta|^2\leq a(\theta)\leq|\theta|^2.
\end{equation}
Consequently $a^{-1}\leq\pi^2/(4|\theta|^2)$, and \eqref{eq:weight-control} follows from \eqref{eq:local-hardy}.
\end{proof}

\subsection{Sine transform and form duality}

Let $T=-\Delta_D$ be the Dirichlet Laplacian on $L^2(\Om)$. Its form domain is $H_0^1(\Om)$, and its eigenvalues in the product sine basis are $|n|^2$. Define the unitary sine transform by
\begin{equation}\label{eq:sine-transform}
 (\mathcal Su)(\theta)=\left(\frac2\pi\right)^{N/2}
 \sum_{n\in\N^N}u(n)\prod_{j=1}^N\sin(n_j\theta_j).
\end{equation}
The functions $(2/\pi)^{N/2}\prod_j\sin(n_j\theta_j)$ form an orthonormal basis of $L^2(\Om)$, so $\mathcal S$ extends from finite sums to a unitary map. The identity
\[
 \sin((n+1)t)+\sin((n-1)t)=2\cos t\sin(nt),\qquad \sin(0t)=0,
\]
accounts for the zero boundary value in each coordinate. Since
$2-2\cos t=4\sin^2(t/2)$, it gives the first identity below. The second follows because the product sine function indexed by $n$ is a Dirichlet eigenfunction with eigenvalue $|n|^2$:
\begin{equation}\label{eq:transformed-operators}
 \mathcal SL_+\mathcal S^*=M_a,
 \qquad
 \mathcal SM_{|n|^{-2}}\mathcal S^*=T^{-1}.
\end{equation}
Both identities hold first on finite sums and then everywhere by boundedness. In particular, no inverse of the multiplier $a$ has been used at this stage.

The following proposition handles the unbounded multiplier $a^{-1/2}$ explicitly.

\begin{proposition}\label{prop:duality}
The operator
\[
 B=M_{a^{-1/2}}T^{-1/2}:L^2(\Om)\longrightarrow L^2(\Om)
\]
is bounded, and
\begin{equation}\label{eq:duality-norm}
 \Cp=\|B\|^{-2}
 =\inf_{0\ne v\in H_0^1(\Om)}\frac{\qform[v]}{\wform[v]}.
\end{equation}
\end{proposition}

\begin{proof}
There are two points to check: the product defining $B$ is bounded despite the singular multiplier, and exchanging the two operators preserves the best coefficient.

For $g\in L^2(\Om)$ set $v=T^{-1/2}g$. The spectral expansion in the sine basis gives $v\in H_0^1(\Om)$ and
\[
 \qform[v]=\|g\|_2^2.
\]
Moreover every $v\in H_0^1(\Om)$ arises this way, with $g=T^{1/2}v$. Lemma~\ref{lem:local-hardy} therefore implies
\[
 \|M_{a^{-1/2}}T^{-1/2}g\|_2^2
 =\wform[v]\leq\frac{\pi^2}{4A_N}\|g\|_2^2.
\]
This proves that the composition is defined for every $g$ and is bounded. Taking the supremum over $g$ also proves
\begin{equation}\label{eq:B-norm}
 \|B\|^2=\sup_{0\ne v\in H_0^1(\Om)}\frac{\wform[v]}{\qform[v]}.
\end{equation}

Fix $c>0$. By \eqref{eq:transformed-operators}, the inequality $\mathcal E_+(u)\geq c\mathcal D_+(u)$ is equivalent, on finite sine sums, to
\begin{equation}\label{eq:bounded-form-inequality}
 c\|T^{-1/2}f\|_2^2\leq\|a^{1/2}f\|_2^2.
\end{equation}
The operators $T^{-1/2}$ and $M_{a^{1/2}}$ are bounded, so the inequality holds on finite sums if and only if it holds for every $f\in L^2(\Om)$.

For $f,g\in L^2(\Om)$, cancellation of the multipliers inside the inner product gives
\[
 \langle Bg,a^{1/2}f\rangle
 =\langle T^{-1/2}g,f\rangle
 =\langle g,T^{-1/2}f\rangle.
\]
Thus $B^*(a^{1/2}f)=T^{-1/2}f$. The range of $M_{a^{1/2}}$ is dense: a function orthogonal to that range satisfies $a^{1/2}g=0$, and hence is zero since $a>0$ almost everywhere. It follows that \eqref{eq:bounded-form-inequality} is equivalent to
\[
 c\|B^*z\|_2^2\leq\|z\|_2^2
 \quad\text{for every }z\in L^2(\Om),
\]
first on the dense range and then by continuity. This holds exactly when $c\|B^*\|^2\leq1$. Since $\|B^*\|=\|B\|>0$, the largest admissible $c$ is $\|B\|^{-2}$. Combining this fact with \eqref{eq:B-norm} proves both equalities in \eqref{eq:duality-norm}.
\end{proof}

We record why this quotient defines a self-adjoint operator in the weighted space $\mathcal H_N=L^2(\Om,a^{-1}\dd\theta)$. Smooth functions supported in the interior are dense in $\mathcal H_N$: first truncate a weighted square-integrable function to compact subsets of $\Om$, and then use ordinary smooth approximation where the weight is bounded above and below. The estimate \eqref{eq:weight-control} embeds $H_0^1(\Om)$ continuously into $\mathcal H_N$. Together with Poincar\'e's inequality, it makes the norm
\[
 \bigl(\qform[v]+\|v\|_{\mathcal H_N}^2\bigr)^{1/2}
\]
equivalent to the usual $H_0^1$ norm. The form is therefore densely defined and closed. The representation theorem for closed nonnegative forms gives its self-adjoint operator $K_N$; we use the standard form and variational framework described in \cite{ReedSimon}. A function $v\in H_0^1(\Om)$ belongs to $\Dom(K_N)$ precisely when there is $f\in\mathcal H_N$ such that
\begin{equation}\label{eq:K-domain}
 \int_{\Om}\nabla v\cdot\overline{\nabla\psi}\dd\theta
 =\int_{\Om}\frac{f\overline\psi}{a}\dd\theta
 \qquad(\psi\in H_0^1(\Om)).
\end{equation}
In this case $K_Nv=f$. Thus $K_N=-a\Delta$ in the interior, and its eigenvalue equation is
\begin{equation}\label{eq:dual-euler}
                         -\Delta v=\frac\lambda a v.
\end{equation}

\section{The critical profile and the dimensional transition}\label{sec:transition}

The critical continuum profile identifies the coefficient $A_N$ but does not by itself decide the lattice inequality. In the cube formulation, its residual separates these two roles. We compute the residual pointwise, then use a tangential integration identity to determine its total sign.

\subsection{Conjugation by the Dirichlet ground state}

Introduce the coordinates and their first two power sums
\begin{equation}\label{eq:q-coordinates}
 q_j=\sin^2\frac{\theta_j}{2},\qquad
 S=\sum_{j=1}^Nq_j,\qquad Q_2=\sum_{j=1}^Nq_j^2,
 \qquad \tau=\frac{Q_2}{S^2}.
\end{equation}
Thus $a=4S$, $0<q_j<1$, and $1/N\leq\tau\leq1$. The function $h=\prod_j\sin\theta_j$ satisfies $-\Delta h=Nh$. Conjugating by $h$ gives the differential operator
\begin{equation}\label{eq:Jacobi}
 \mathcal J_N=\sum_{j=1}^N
 \left[q_j(1-q_j)\partial_j^2+
       \left(\frac32-3q_j\right)\partial_j\right].
\end{equation}
For completeness, the derivatives of one coordinate are
\[
 \frac{\mathrm dq_j}{\mathrm d\theta_j}=\frac12\sin\theta_j
 =\sqrt{q_j(1-q_j)},\qquad
 \frac{\mathrm d^2q_j}{\mathrm d\theta_j^2}=\frac12-q_j.
\]
The product rule gives
\[
 h^{-1}\Delta(hf)=-Nf+\Delta f
                  +2\sum_j\cot\theta_j\,\partial_{\theta_j}f.
\]
In the $q$ coordinates the second derivative contributes
$q_j(1-q_j)\partial_j^2+(1/2-q_j)\partial_j$, and the last term contributes
$(1-2q_j)\partial_j$. Adding them yields
\begin{equation}\label{eq:conjugation}
 -\Delta_\theta(hf(q))=h(N-\mathcal J_N)f(q).
\end{equation}
In particular,
\begin{equation}\label{eq:J-radial}
 \mathcal J_N F(S)=(S-Q_2)F''(S)
                    +\left(\frac{3N}{2}-3S\right)F'(S).
\end{equation}

The critical exponent makes the coefficient of $a^{-1}$ equal to its continuum value.

\begin{lemma}\label{lem:critical-residual}
Let
\begin{equation}\label{eq:critical-profile}
 \beta_N=\frac{3N-2}{4},\qquad
 \Phi_N(\theta)=h(\theta)a(\theta)^{-\beta_N}.
\end{equation}
Then $\Phi_N>0$ in $\Om$ and
\begin{equation}\label{eq:critical-residual}
 \frac{-\Delta\Phi_N}{\Phi_N}=\frac{A_N}{a}+R_N,
 \qquad
 R_N=\frac{6-5N}{4}+\frac{9N^2-4}{16}\tau.
\end{equation}
Moreover, $R_N$ is bounded and $\Phi_N\in L^2(\Om)$.
\end{lemma}

\begin{proof}
Let $\beta\in\R$. Since
$(S^{-\beta})'=-\beta S^{-\beta-1}$ and
$(S^{-\beta})''=\beta(\beta+1)S^{-\beta-2}$,
formula \eqref{eq:J-radial} gives
\begin{align*}
 \frac{(N-\mathcal J_N)S^{-\beta}}{S^{-\beta}}
 &=N-\beta(\beta+1)\frac{S-Q_2}{S^2}
       +\frac{\beta}{S}\left(\frac{3N}{2}-3S\right)\\
 &=\frac{\beta(3N/2-\beta-1)}{S}
       +N-3\beta+\beta(\beta+1)\tau.
\end{align*}
At $\beta=\beta_N$, we have $3N/2-\beta_N-1=\beta_N$ and
\[
 4\beta_N^2=A_N,\qquad
 N-3\beta_N=\frac{6-5N}{4},\qquad
 \beta_N(\beta_N+1)=\frac{9N^2-4}{16}.
\]
Use $a=4S$ and \eqref{eq:conjugation} to obtain \eqref{eq:critical-residual}. The inequalities $Q_2\leq S^2\leq NQ_2$ give $1/N\leq\tau\leq1$, and hence boundedness of $R_N$. Positivity of $\Phi_N$ follows from positivity of every sine factor and of $a$ in $\Om$.

We make the corner estimate explicit. Let $r=|\theta|$ and $\omega=\theta/r$. Taylor's formula, uniformly for $\omega\in\Sigma_N$, gives
\[
 a(r\omega)=r^2\bigl(1+O(r^2)\bigr),\qquad
 h(r\omega)=r^NY_N(\omega)\prod_j\frac{\sin(r\omega_j)}{r\omega_j}
 =r^NY_N(\omega)\bigl(1+O(r^2)\bigr).
\]
The quotient $\sin t/t$ is interpreted as $1$ at $t=0$, so the estimate remains uniform up to the angular boundary. Since $N-2\beta_N=-(N-2)/2$, it follows that
\begin{equation}\label{eq:critical-asymptotic}
 \Phi_N(r\omega)
 =r^{-(N-2)/2}Y_N(\omega)\bigl(1+O(r^2)\bigr).
\end{equation}
In particular, there are $c,C,r_0>0$ such that
\begin{equation}\label{eq:critical-mass-bounds}
 cr\leq\int_{\Sigma_N}\Phi_N(r\omega)^2r^{N-1}\dd\omega\leq Cr
 \quad(0<r<r_0).
\end{equation}
For $N=1$, use the one-point angular measure of mass one. Integrating the upper bound over $r$ proves square integrability near the corner. Away from the corner, $a$ is bounded below and $h$ is smooth up to the boundary. This completes the proof.
\end{proof}

The next identity converts a positive comparison function into an energy inequality. We give the elementary proof, including the extension that will be used in the first two dimensions. More general ground-state representations are developed in \cite{FrankSeiringer}.

\begin{lemma}\label{lem:ground-state}
Let $\Phi\in C^2(\Om)$ be positive and set $V=-\Delta\Phi/\Phi$. For every $v\in C_c^\infty(\Om)$,
\begin{equation}\label{eq:ground-state}
 \int_{\Om}|\nabla v|^2\dd\theta-
 \int_{\Om}V|v|^2\dd\theta
 =\int_{\Om}\Phi^2\left|\nabla\left(\frac v\Phi\right)\right|^2\dd\theta.
\end{equation}
If $V=c/a+R$ with $c\in\R$ and $R\in L^\infty(\Om)$, then
\begin{equation}\label{eq:ground-state-extension}
 \qform[v]-c\wform[v]\geq\int_{\Om}R|v|^2\dd\theta
 \quad(v\in H_0^1(\Om)).
\end{equation}
\end{lemma}

\begin{proof}
Write $v=\Phi f$. Since $f$ has compact support in the interior,
\[
 |\nabla(\Phi f)|^2
 =\Phi^2|\nabla f|^2+|f|^2|\nabla\Phi|^2
       +\Phi\nabla\Phi\cdot\nabla|f|^2.
\]
The last two terms together are $\nabla\Phi\cdot\nabla(\Phi|f|^2)$. Their integral equals
$-\int\Delta\Phi\,\Phi|f|^2=\int V|v|^2$, with no boundary term. This proves \eqref{eq:ground-state}, also for complex $v$.

Drop the nonnegative right-hand side and take $v_m\in C_c^\infty(\Om)$ converging to $v$ in $H_0^1$. The Dirichlet energies converge, and so do the weighted norms by \eqref{eq:weight-control}. Since $R$ is bounded and $v_m\to v$ in $L^2$, the integrals against $R$ converge as well. Passing to the limit gives \eqref{eq:ground-state-extension}.
\end{proof}

\subsection{Logarithmic cutoffs and the continuum upper bound}

The profile $\Phi_N$ has finite unweighted mass and infinite weighted mass. This permits an arbitrarily small cutoff cost while retaining the sign of its integrated residual.

\begin{lemma}\label{lem:cutoffs}
For every $N\geq1$ there are functions $v_\delta\in H_0^1(\Om)$, $0<\delta<1$, such that
\begin{align}
 \wform[v_\delta]&\longrightarrow\infty,\label{eq:cutoff-mass}\\
 \qform[v_\delta]-A_N\wform[v_\delta]
 &\longrightarrow\int_{\Om}R_N\Phi_N^2\dd\theta.
 \label{eq:cutoff-defect}
\end{align}
Consequently $\Cp\leq A_N$ for every $N$.
\end{lemma}

\begin{proof}
Put $\ell_\delta=\log(1/\delta)$, and define
\begin{equation}\label{eq:log-cutoff}
 \chi_\delta(S)=
 \begin{cases}
 0,&S\leq\delta^2,\\
 \displaystyle\frac{\log(S/\delta^2)}{\ell_\delta},&\delta^2<S<\delta,\\
 1,&S\geq\delta,
 \end{cases}
 \qquad v_\delta=\chi_\delta\Phi_N.
\end{equation}
For each fixed $\delta$, this is an $H^1$ function which vanishes near the only singular corner. Elsewhere $\Phi_N$ is smooth up to the faces and has zero trace because of the factor $h$. The trace characterization on the cube gives $v_\delta\in H_0^1(\Om)$.

We first prove divergence of its weighted norm. By \eqref{eq:a-comparison},
$S\geq r^2/\pi^2$. Hence $\chi_\delta=1$ when $r\geq\pi\sqrt\delta$. On the fixed small ball $r<r_0$, the estimates \eqref{eq:a-comparison} and \eqref{eq:critical-mass-bounds} yield
\[
 \wform[v_\delta]\geq c\int_{\pi\sqrt\delta}^{r_0}\frac{\dd r}{r}
 \longrightarrow\infty.
\]
We take $\delta$ small enough that the lower endpoint is less than $r_0$.

Next test the equation
$-\Delta\Phi_N=(A_N/a+R_N)\Phi_N$ against $\chi_\delta^2\Phi_N$. This test belongs to $H_0^1(\Om)$ and vanishes near the corner. To justify it, multiply $\Phi_N$ by a smooth cutoff equal to one on $S\geq\delta^2$ and zero on $S\leq\delta^2/2$. The resulting function is smooth on the closed cube and belongs to $H_0^1$; its weak equation agrees with that of $\Phi_N$ on the support of the test. Density of compactly supported smooth tests therefore justifies the integration by parts. Expanding the two gradients gives
\begin{equation}\label{eq:cutoff-identity}
 \qform[v_\delta]-A_N\wform[v_\delta]
 =\int_{\Om}\Phi_N^2|\nabla\chi_\delta|^2\dd\theta
  +\int_{\Om}R_N\Phi_N^2\chi_\delta^2\dd\theta.
\end{equation}
The Lipschitz junctions in $\chi_\delta$ are harmless for this weak identity; alternatively they can be smoothed in $H^1$ with $\delta$ fixed.

On the transition region,
\[
 |\nabla\chi_\delta|^2
 =\frac{|\nabla S|^2}{S^2\ell_\delta^2}
 =\frac{S-Q_2}{S^2\ell_\delta^2}
 \leq\frac1{S\ell_\delta^2}.
\]
The comparison $r^2/\pi^2\leq S\leq r^2/4$ shows that
$\{\delta^2<S<\delta\}\subset\{2\delta<r<\pi\sqrt\delta\}$.
For small $\delta$, this annulus lies in $r<r_0$, and consequently
\begin{align*}
 \int_{\Om}\Phi_N^2|\nabla\chi_\delta|^2\dd\theta
 &\leq\frac{C}{\ell_\delta^2}
          \int_{2\delta}^{\pi\sqrt\delta}\frac{\dd r}{r}\\
 &=\frac{C}{\ell_\delta^2}
       \left(\frac12\ell_\delta+\log\frac\pi2\right)
 \longrightarrow0.
\end{align*}
Finally, $0\leq\chi_\delta\leq1$ and $\chi_\delta\to1$ pointwise in $\Om$. The function $|R_N|\Phi_N^2$ is integrable by Lemma~\ref{lem:critical-residual}, so dominated convergence applies to the second term of \eqref{eq:cutoff-identity}. This proves \eqref{eq:cutoff-defect}. The numerator of
\[
 \frac{\qform[v_\delta]}{\wform[v_\delta]}
 =A_N+\frac{\qform[v_\delta]-A_N\wform[v_\delta]}{\wform[v_\delta]}
\]
in the second fraction remains bounded while its denominator tends to infinity. The quotient tends to $A_N$, and Proposition~\ref{prop:duality} proves $C_+(N)\leq A_N$.
\end{proof}

\subsection{The exact values in dimensions one and two}

In the first two dimensions, the residual has a pointwise sign. It gives both the lower bound and nonattainment of the continuous quotient.

\begin{proposition}\label{prop:low-dimensions}
One has $C_+(1)=1/4$ and $C_+(2)=4$. Neither infimum in \eqref{eq:duality-norm} is attained by a nonzero function.
\end{proposition}

\begin{proof}
When $N=1$, the definitions give $S=q_1$, $\tau=1$ and
\[
 R_1=\frac{6-5}{4}+\frac{9-4}{16}=\frac9{16}.
\]
Apply \eqref{eq:ground-state-extension} with $\Phi=\Phi_1$ and $c=A_1=1/4$. It follows that
\[
 \mathfrak q_1[v]-\frac14\mathfrak w_1[v]
 \geq\frac9{16}\|v\|_2^2
 \quad(v\in H_0^1(0,\pi)).
\]
Thus $C_+(1)\geq1/4$, and equality in the quotient is impossible for a nonzero $v$.

When $N=2$, the residual is $R_2=-1+2\tau$. Using $S=q_1+q_2$, we obtain
\begin{equation}\label{eq:two-remainder}
 R_2=\frac{2(q_1^2+q_2^2)-(q_1+q_2)^2}{S^2}
 =\frac{(q_1-q_2)^2}{S^2}
 =\frac{4(\cos\theta_1-\cos\theta_2)^2}{a^2}.
\end{equation}
The same lemma gives
\[
 \mathfrak q_2[v]-4\mathfrak w_2[v]
 \geq\int_{\Omega_2}R_2|v|^2\dd\theta.
\]
The function $t\mapsto\sin^2(t/2)$ is strictly increasing on $(0,\pi)$, so $R_2$ vanishes exactly on $\theta_1=\theta_2$. This diagonal has two-dimensional measure zero. Therefore the last integral is strictly positive whenever $v\ne0$ in $L^2$. We conclude that $C_+(2)\geq4$ and that the infimum is not attained. In both dimensions Lemma~\ref{lem:cutoffs} gives the opposite inequality, proving the stated exact values.
\end{proof}

\subsection{A tangential identity for the integrated residual}

For higher dimensions the pointwise residual can change sign. Its integral is nevertheless determined by derivatives that preserve $S$. Define
\begin{equation}\label{eq:anisotropy}
 \mathcal K(q)=\sum_{1\leq i<j\leq N}
 \frac{q_iq_j(q_i-q_j)^2}{S^2(1-q_i)(1-q_j)}.
\end{equation}
This is nonnegative. Its apparent singularities at $q_i=1$ are integrable against the densities used below.

\begin{lemma}\label{lem:tangential}
Let $\gamma<3N/2$ and
\begin{equation}\label{eq:rho}
 \rho_\gamma(q)=S^{-\gamma}\prod_{j=1}^N\sqrt{q_j(1-q_j)}
 \quad(q\in(0,1)^N).
\end{equation}
Then $\rho_\gamma$, $\tau\rho_\gamma$ and $\mathcal K\rho_\gamma$ are integrable, and
\begin{equation}\label{eq:tangential-identity}
 \int_{(0,1)^N}\bigl[5-(3N+2)\tau-\mathcal K\bigr]
                        \rho_\gamma\dd q=0.
\end{equation}
\end{lemma}

\begin{proof}
If $N=1$, then $\rho_\gamma=q_1^{1/2-\gamma}(1-q_1)^{1/2}$ is integrable because $\gamma<3/2$. Also $\tau=1$ and $\mathcal K=0$, so the integrand in \eqref{eq:tangential-identity} vanishes. Assume henceforth that $N\geq2$.

We first verify the integrability needed for integration by parts. When $S<1/2$, write $q_j=Sz_j$, where $z_j>0$ and $\sum_jz_j=1$. With $z_1,\ldots,z_{N-1}$ as independent coordinates, the Jacobian is $S^{N-1}$. Thus
\[
 \rho_\gamma\dd q
 =S^{3N/2-\gamma-1}
   \prod_j z_j^{1/2}(1-Sz_j)^{1/2}\dd S\dd z.
\]
The integral over the simplex is finite, and $3N/2-\gamma-1>-1$. Moreover $\tau\leq1$, and each summand of $\mathcal K$ is bounded by $CS^2$ in this region. These facts prove integrability near $S=0$. Where $S\geq1/2$, all factors involving $S$ are bounded. In a term $\mathcal K_{ij}\rho_\gamma$, the possible singular factors at the upper faces are
$(1-q_i)^{-1/2}(1-q_j)^{-1/2}$; all remaining endpoint powers are nonnegative. These factors are integrable, including at intersections of faces. This proves all three integrability assertions.

Fix $i<j$, and define
\[
 D_{ij}=\partial_i-\partial_j,\qquad
 V_{ij}=\frac{q_iq_j(q_i-q_j)}{S^2},\qquad
 \mathcal K_{ij}=\frac{q_iq_j(q_i-q_j)^2}{S^2(1-q_i)(1-q_j)}.
\]
Since $D_{ij}S=0$, differentiating the radial factor produces no term. The remaining derivatives are
\begin{align*}
 D_{ij}V_{ij}&=\frac{4q_iq_j-q_i^2-q_j^2}{S^2},\\
 D_{ij}\log\rho_\gamma
 &=\frac1{2q_i}-\frac1{2q_j}
   -\frac1{2(1-q_i)}+\frac1{2(1-q_j)}.
\end{align*}
Multiplication by $V_{ij}$ gives, respectively, the two additional contributions
\[
 -\frac{(q_i-q_j)^2}{2S^2},\qquad
 -\frac{q_iq_j(q_i-q_j)^2}{2S^2(1-q_i)(1-q_j)}.
\]
Adding all terms proves the pointwise identity
\begin{equation}\label{eq:pair-divergence}
 \frac{D_{ij}(V_{ij}\rho_\gamma)}{\rho_\gamma}
 =\frac{5q_iq_j-\frac32(q_i^2+q_j^2)}{S^2}
       -\frac12\mathcal K_{ij}.
\end{equation}

We now justify that its integral is zero. Choose $\eta_\varepsilon\in C^\infty([0,\infty))$ with values in $[0,1]$, zero on $S\leq\varepsilon$ and one on $S\geq2\varepsilon$. Then $D_{ij}\eta_\varepsilon(S)=0$. For fixed $\varepsilon$, the function
$\eta_\varepsilon V_{ij}\rho_\gamma$ and its coordinate derivatives in $q_i,q_j$ are integrable. Its endpoint traces in these coordinates are zero: the lower endpoint has a factor $q_i^{3/2}$ or $q_j^{3/2}$, and the upper endpoint has a factor $(1-q_i)^{1/2}$ or $(1-q_j)^{1/2}$. Fubini's theorem and the one-dimensional fundamental theorem of calculus therefore give
\[
 \int_{(0,1)^N}D_{ij}(\eta_\varepsilon V_{ij}\rho_\gamma)\dd q=0.
\]
Equivalently, the integral of the right-hand side of \eqref{eq:pair-divergence}, multiplied by $\eta_\varepsilon\rho_\gamma$, is zero. The absolute value is bounded by $C(1+\mathcal K_{ij})\rho_\gamma$, which is integrable by the first part of the proof. Dominated convergence removes $\eta_\varepsilon$.

Finally, summing over $i<j$ and using
\[
 \sum_{i<j}q_iq_j=\frac{S^2-Q_2}{2},\qquad
 \sum_{i<j}(q_i^2+q_j^2)=(N-1)Q_2,
\]
we find that the sum of the right-hand sides of \eqref{eq:pair-divergence} is
\[
 \frac52(1-\tau)-\frac32(N-1)\tau-\frac12\mathcal K
 =\frac12\bigl[5-(3N+2)\tau-\mathcal K\bigr].
\]
Its integral against $\rho_\gamma$ is zero, proving \eqref{eq:tangential-identity}.
\end{proof}

We now apply the identity at the critical exponent. The outcome isolates a dimension-dependent scalar and a nonnegative correction.

\begin{proposition}\label{prop:total-defect}
For every $N\geq1$,
\begin{equation}\label{eq:total-defect}
 \int_{\Om}R_N\Phi_N^2\dd\theta
 =\frac{14-5N}{16}\int_{\Om}\Phi_N^2\dd\theta
  -\frac{3N-2}{16}\int_{\Om}\mathcal K\Phi_N^2\dd\theta.
\end{equation}
In particular, the left-hand side is strictly negative for $N\geq3$.
\end{proposition}

\begin{proof}
The coordinate change $q_j=\sin^2(\theta_j/2)$ gives
\[
 \dd\theta_j=\frac{\dd q_j}{\sqrt{q_j(1-q_j)}},\qquad
 h^2=4^N\prod_j q_j(1-q_j).
\]
Since $a=4S$, we obtain the exact measure identity
\begin{equation}\label{eq:critical-density}
 \Phi_N^2\dd\theta
 =4^{N-2\beta_N}S^{-2\beta_N}
       \prod_j\sqrt{q_j(1-q_j)}\dd q.
\end{equation}
Here $2\beta_N=3N/2-1<3N/2$, so Lemma~\ref{lem:tangential} applies and all resulting integrals are finite. Introduce the abbreviations
\[
 I_0=\int_{\Om}\Phi_N^2\dd\theta,\quad
 I_\tau=\int_{\Om}\tau\Phi_N^2\dd\theta,\quad
 I_K=\int_{\Om}\mathcal K\Phi_N^2\dd\theta.
\]
After multiplying \eqref{eq:tangential-identity} by the constant in \eqref{eq:critical-density}, it reads
$(3N+2)I_\tau=5I_0-I_K$. Substitution into \eqref{eq:critical-residual} yields
\begin{align*}
 \int_{\Om}R_N\Phi_N^2\dd\theta
 &=\frac{6-5N}{4}I_0
    +\frac{9N^2-4}{16(3N+2)}(5I_0-I_K)\\
 &=\frac{4(6-5N)+5(3N-2)}{16}I_0
    -\frac{3N-2}{16}I_K\\
 &=\frac{14-5N}{16}I_0-\frac{3N-2}{16}I_K.
\end{align*}
This is \eqref{eq:total-defect}. We have $I_0>0$ because $\Phi_N>0$, and $I_K\geq0$ because $\mathcal K\geq0$. For $N\geq3$, $14-5N\leq-1$, so the first term is strictly negative and the second is nonpositive.
\end{proof}

\begin{proof}[Proof of \eqref{eq:classification} in Theorem~\ref{thm:main}]
Proposition~\ref{prop:low-dimensions} gives $C_+(1)=A_1=1/4$ and $C_+(2)=A_2=4$. Fix $N\geq3$ and write
\[
 d_N=\int_{\Om}R_N\Phi_N^2\dd\theta<0.
\]
By Lemma~\ref{lem:cutoffs}, the quantity $\qform[v_\delta]-A_N\wform[v_\delta]$ tends to $d_N$. Choose $\delta$ so small that its difference from $d_N$ has absolute value less than $-d_N/2$. Then
\[
 \qform[v_\delta]-A_N\wform[v_\delta]<d_N/2<0.
\]
The test function is nonzero and its weighted norm is positive and finite. Proposition~\ref{prop:duality} gives
\[
 C_+(N)\leq\frac{\qform[v_\delta]}{\wform[v_\delta]}<A_N.
\]
This proves the entire dimensional classification and, in particular, the upper bound $C_+(3)<49/4$. The lower bound in \eqref{eq:three-bounds} will be proved in Section~\ref{sec:three} after attainment has been established.
\end{proof}

In dimension three the identity has the particularly simple form
\begin{equation}\label{eq:three-total-defect}
 \int_{\Omega_3}R_3\Phi_3^2\dd\theta
 =-\frac1{16}\int_{\Omega_3}\Phi_3^2\dd\theta
  -\frac7{16}\int_{\Omega_3}\mathcal K\Phi_3^2\dd\theta<0.
\end{equation}
The local singularity still determines the coefficient $49/4$. The strict decrease comes from the integrated residual over the cube, after the singular cutoff cost has vanished.

\section{The spectral threshold and attainment}\label{sec:spectrum}

The strict inequality proved above creates an eigenvalue below a fixed essential threshold. To identify that threshold, we localize near the singular corner and transform its radial variable logarithmically. The remainder of the cube contributes only compactness. This also clarifies why the positive remainders in dimensions one and two do not raise the best coefficient.

\subsection{The essential spectrum of the weighted operator}

The following result identifies the essential spectrum through localization and explicit Weyl sequences.

\begin{proposition}\label{prop:essential}
For the form realization defined in \eqref{eq:K-domain},
\begin{equation}\label{eq:essential-spectrum}
                       \sigma_{\ess}(K_N)=[A_N,\infty).
\end{equation}
\end{proposition}

\begin{proof}
We prove the two spectral inclusions separately. The variational principle, spectral projections and singular Weyl criterion used below are standard; see \cite{ReedSimon}. We supply the estimates specific to the singular weight.

\emph{Exclusion of essential spectrum below $A_N$.}
Taylor's formula for $4\sin^2(\theta_j/2)$, summed over the coordinates, gives
\begin{equation}\label{eq:a-taylor}
 a(r\omega)=r^2+O(r^4),\qquad \frac{r^2}{a(r\omega)}=1+O(r^2),
\end{equation}
uniformly in $\omega\in\Sigma_N$. Fix $0<\varepsilon<A_N$. Choose a small $r_0>0$ such that
$a^{-1}\leq A_N(A_N-\varepsilon)^{-1}r^{-2}$ on $r<2r_0$. Lemma~\ref{lem:local-hardy} implies
\begin{equation}\label{eq:local-threshold}
 \qform[v]\geq(A_N-\varepsilon)\wform[v]
 \quad\text{if }\supp v\subset\{r<2r_0\}.
\end{equation}

Choose real smooth radial cutoffs $\chi,\eta$ with $\chi^2+\eta^2=1$, $\chi=1$ on $r\leq r_0$, and $\chi=0$ on $r\geq2r_0$. Such a pair is obtained by taking the cosine and sine of a smooth transition from $0$ to $\pi/2$. The product rule gives the localization identity
\[
 \qform[v]=\qform[\chi v]+\qform[\eta v]
 -\int_{\Om}(|\nabla\chi|^2+|\nabla\eta|^2)|v|^2\dd\theta.
\]
The mixed terms cancel because $\chi\nabla\chi+\eta\nabla\eta=0$. On the support of $\eta$, the weight $a^{-1}$ is bounded. Thus \eqref{eq:local-threshold}, $\qform[\eta v]\geq0$, and
$\wform[v]=\wform[\chi v]+\wform[\eta v]$ imply
\begin{equation}\label{eq:compact-error}
 \qform[v]\geq(A_N-\varepsilon)\|v\|_{\mathcal H_N}^2
                        -C_\varepsilon\|v\|_2^2.
\end{equation}

Let $0\leq\lambda<A_N-\varepsilon$, and suppose that the range of the spectral projection $\mathbf1_{[0,\lambda]}(K_N)$ is infinite dimensional. It contains an orthonormal sequence $(u_j)$ in $\mathcal H_N$. The spectral theorem gives
$\qform[u_j]\leq\lambda$, so this sequence is bounded in $H_0^1(\Om)$. Rellich compactness supplies a subsequence converging strongly in unweighted $L^2$. The limit is zero. Indeed $u_j\rightharpoonup0$ in $\mathcal H_N$, and, for each $g\in L^2(\Om)$,
\[
 ag\in\mathcal H_N,\qquad
 \langle u_j,g\rangle_{L^2}=\langle u_j,ag\rangle_{\mathcal H_N}\longrightarrow0.
\]
Applying \eqref{eq:compact-error} to that subsequence gives
$\lambda\geq A_N-\varepsilon-o(1)$, a contradiction. Hence every spectral projection onto an interval strictly below $A_N$ has finite rank, after taking $\varepsilon$ sufficiently small. Since $K_N\geq0$, this proves
$\sigma_{\ess}(K_N)\subset[A_N,\infty)$.

\emph{Construction of singular Weyl sequences.}
Fix $r_0<\pi$ and restrict attention to $0<r<r_0$, where the cube agrees exactly with the cone of angular section $\Sigma_N$. Set
\[
 b_N=\frac{N-2}{2},\qquad t=-\log r,\qquad
 v(r\omega)=r^{-b_N}w(t,\omega),\qquad
 m(t,\omega)=\frac{r^2}{a(r\omega)}.
\]
For a function supported in a compact subinterval of $t> -\log r_0$, the change of variables gives
\begin{align}
 \|v\|_{\mathcal H_N}^2
 &=\int m|w|^2\dd t\dd\omega,
 \label{eq:cylinder-norm}\\
 \qform[v]
 &=\int\bigl(|\partial_tw|^2+|\nabla_\omega w|^2+b_N^2|w|^2\bigr)
                  \dd t\dd\omega.
 \label{eq:cylinder-form}
\end{align}
For the norm, the radial factor is
$r^{-2b_N}r^{N-1}\dd r/a=(r^2/a)\dd t$ after reversing the integration limits. For the energy,
$\partial_rv=-r^{-b_N-1}(b_Nw+\partial_tw)$; the cross term is
$b_N\int\partial_t|w|^2\dd t=0$. This proves both identities.

Write $A_0=-\partial_t^2-\Delta_{\Sigma_N}+b_N^2$. Comparing the form in \eqref{eq:cylinder-form} with the weighted inner product in \eqref{eq:cylinder-norm} shows that the transformed operator is
\begin{equation}\label{eq:cylinder-operator}
 m^{-1}A_0,
 \qquad m=1+O(e^{-2t}),
\end{equation}
with Dirichlet angular conditions. There are no angular terms for $N=1$.

Normalize $Y_N$ to have angular $L^2$ norm one; in dimension one use the one-point measure. Given $\lambda\geq A_N$, put $k=\sqrt{\lambda-A_N}$. Choose $\zeta\in C_c^\infty(0,1)$ with $\|\zeta\|_2=1$, and define
\[
 w_j(t,\omega)=j^{-1/2}\zeta\!\left(\frac{t-j^2}{j}\right)e^{ikt}Y_N(\omega).
\]
For large $j$ its radial support lies in the chosen cone. Since $\mu_N+b_N^2=A_N$, direct differentiation yields
\[
 (A_0-\lambda)w_j
 =e^{ikt}Y_N\left[-j^{-5/2}\zeta''\!\left(\frac{t-j^2}{j}\right)
                   -2ikj^{-3/2}\zeta'\!\left(\frac{t-j^2}{j}\right)\right].
\]
Consequently
\begin{equation}\label{eq:weyl-error}
 \|(A_0-\lambda)w_j\|_2
 \leq j^{-2}\|\zeta''\|_2+2|k|j^{-1}\|\zeta'\|_2\longrightarrow0.
\end{equation}
This also covers $\lambda=A_N$, when $k=0$.

Let $v_j$ denote the corresponding function on the cube, extended by zero outside the small cone. It belongs to $\Dom(K_N)$. To check this, note that its radial cutoff vanishes near both endpoints of its annular support. Furthermore $r^{-b_N}Y_N(\omega)$ is a constant multiple of $r^{-b_N-N}\prod_i\theta_i$, which is smooth away from the origin and vanishes on each coordinate face. Thus $v_j$ is smooth on its closed annulus, has zero trace, and $-a\Delta v_j\in\mathcal H_N$. Integration by parts verifies \eqref{eq:K-domain}.

The unweighted cylinder norm of $w_j$ is one, and
$\sup_{\supp w_j}|m-1|\leq Ce^{-2j^2}$, so $\|v_j\|_{\mathcal H_N}\to1$. Also
\[
 (m^{-1}A_0-\lambda)w_j
 =m^{-1}(A_0-\lambda)w_j+\lambda(m^{-1}-1)w_j.
\]
The weighted and unweighted cylinder norms are uniformly equivalent on these supports. Combining this fact with \eqref{eq:weyl-error} gives
$\|(K_N-\lambda)v_j\|_{\mathcal H_N}\to0$.
Finally the supports lie in $r\leq e^{-j^2}$, and for every $g\in\mathcal H_N$,
\[
 |\langle v_j,g\rangle_{\mathcal H_N}|
 \leq\|v_j\|_{\mathcal H_N}
 \left(\int_{\{r\leq e^{-j^2}\}}\frac{|g|^2}{a}\dd\theta\right)^{1/2}
 \longrightarrow0.
\]
After normalizing $v_j$, the singular Weyl criterion places $\lambda$ in the essential spectrum. Since $\lambda\geq A_N$ was arbitrary, the reverse inclusion follows.
\end{proof}

\subsection{The positive eigenfunction below the threshold}

The spectral gap between $\Cp$ and $A_N$ prevents loss of compactness at the corner. Positivity and connectedness then determine the first eigenspace.

\begin{corollary}\label{cor:attainment}
For every $N\geq3$, the infimum in \eqref{eq:duality-norm} is attained by a function $v_N\in H_0^1(\Om)$ satisfying
\begin{equation}\label{eq:first-eigenfunction}
 -\Delta v_N=\frac{\Cp}{a}v_N,
 \qquad v_N>0\text{ in }\Om,
 \qquad \wform[v_N]=1.
\end{equation}
This normalization determines $v_N$ uniquely among positive functions. It is invariant under permutations of the coordinates.
\end{corollary}

\begin{proof}
Proposition~\ref{prop:duality} identifies $C_+(N)$ with the spectral bottom of $K_N$. For $N\geq3$, the classification gives $C_+(N)<A_N$, and Proposition~\ref{prop:essential} places the essential spectrum at and above $A_N$. The spectral bottom must therefore be an isolated eigenvalue of finite multiplicity. Choose an eigenfunction and normalize its $\mathcal H_N$ norm to one. Every operator-domain eigenfunction belongs to the form domain $H_0^1(\Om)$.

The modulus of an $H^1$ function belongs to $H^1$, has the same zero trace, and satisfies $|\nabla|u||\leq|\nabla u|$ almost everywhere. Replacing the eigenfunction by its modulus preserves its weighted norm and cannot increase its energy. Since its quotient already equals the infimum, its modulus also minimizes the quotient. Differentiating that quotient in each direction $\psi\in H_0^1(\Om)$ gives the weak equation
\begin{equation}\label{eq:attainment-weak}
 \int_{\Om}\nabla v_N\cdot\nabla\psi\dd\theta
 =C_+(N)\int_{\Om}\frac{v_N\psi}{a}\dd\theta.
\end{equation}
The weighted pairing is continuous by \eqref{eq:weight-control}, so all these variations are legitimate. This constructs a nonnegative normalized eigenfunction.

On every compact subset of $\Om$, the coefficient $C_+(N)/a$ is smooth and bounded. Interior elliptic regularity, iterated in the equation, gives a smooth representative of $v_N$. Also $-\Delta v_N=(C_+(N)/a)v_N\geq0$, because $C_+(N)>0$ by Proposition~\ref{prop:duality}. The strong maximum principle and connectedness imply $v_N>0$ throughout the interior.

To prove simplicity, take any nonzero real eigenfunction $u$ at the spectral bottom. Its modulus is a minimizer by the same argument, so $|u|$ has a strictly positive smooth representative. The function $u$ is continuous by interior regularity. If it changed sign, a path inside the cube joining points of opposite sign would contain a zero, contradicting positivity of $|u|$. Thus every nonzero real eigenfunction has a fixed sign. If the real eigenspace had dimension at least two, a linear combination could be chosen orthogonal in $\mathcal H_N$ to $v_N$. Such a nonzero function of fixed sign cannot be orthogonal to $v_N>0$. This contradiction proves real simplicity. Since the equation has real coefficients, taking real and imaginary parts proves complex simplicity as well.

A permutation of the coordinates preserves $a$, the cube, the Dirichlet form and the weighted norm. It therefore maps $v_N$ to a positive normalized eigenfunction for the same eigenvalue. Simplicity and normalization make that function equal to $v_N$, proving the final assertion.
\end{proof}

\subsection{The discrete coupling threshold}

The same number $\Cp$ has a different spectral role for the original lattice operator: it is the threshold in the coupling parameter $c$.

\begin{proposition}\label{prop:coupling}
For every $N\geq1$ and $c\in\R$,
\begin{equation}\label{eq:discrete-essential}
 \sigma_{\ess}(H_N(c))=[0,4N],
\end{equation}
and
\begin{equation}\label{eq:negative-iff}
 H_N(c)\text{ has a negative eigenvalue}
                  \quad\Longleftrightarrow\quad c>\Cp.
\end{equation}
In particular, $H_N(A_N)$ is nonnegative for $N=1,2$ and has a negative eigenvalue for every $N\geq3$.
\end{proposition}

\begin{proof}
The sine transform represents $L_+$ by multiplication by $a$ on $L^2(\Om)$. The continuous function $a$ takes values in $(0,4N)$ and its closure of values is $[0,4N]$. Every open interval meeting $(0,4N)$ has a preimage of positive measure, so the essential range is exactly $[0,4N]$. A multiplication operator on this nonatomic space has no isolated eigenvalues of finite multiplicity. Hence
$\sigma(L_+)=\sigma_{\ess}(L_+)=[0,4N]$.

For $R>0$, truncate the multiplier $M_{|n|^{-2}}$ to $|n|\leq R$. The truncated operator has finite rank, and the norm of the difference is at most $R^{-2}$. Thus the multiplier is compact. Weyl's invariance theorem for compact self-adjoint perturbations gives \eqref{eq:discrete-essential} for every real $c$.

If $c\leq C_+(N)$, then for every finitely supported $u$,
\[
 \langle H_N(c)u,u\rangle
 =\mathcal E_+(u)-C_+(N)\mathcal D_+(u)
       +(C_+(N)-c)\mathcal D_+(u)\geq0.
\]
Both operators in $H_N(c)$ are bounded, so density extends the inequality to all of $\ell^2(\N^N)$. Therefore there is no negative spectrum.

If $c>C_+(N)$, the definition of the infimum supplies a nonzero finitely supported $u$ with $\mathcal E_+(u)/\mathcal D_+(u)<c$. Its expectation of $H_N(c)$ is negative. The variational principle implies $\inf\sigma(H_N(c))<0$. All spectrum below zero consists of isolated eigenvalues of finite multiplicity, since the essential spectrum is $[0,4N]$. There is therefore a negative eigenvalue. This proves the equivalence, and the assertion for $c=A_N$ follows from \eqref{eq:classification}.
\end{proof}

\begin{proof}[Proof of Theorem~\ref{thm:spectral-main}]
Proposition~\ref{prop:duality} identifies the discrete coefficient with the weighted Rayleigh infimum. The closed-form construction in Section~\ref{sec:duality} identifies that infimum with $\inf\sigma(K_N)$. Proposition~\ref{prop:essential} gives its essential spectrum. Proposition~\ref{prop:low-dimensions} proves nonattainment in dimensions one and two, whereas Corollary~\ref{cor:attainment} gives attainment, simplicity, positivity and permutation invariance in every higher dimension. Finally Proposition~\ref{prop:coupling} proves both assertions about the original discrete Schr\"odinger operator. These statements establish every part of the theorem.
\end{proof}

The relation between minimizing functions also involves form duality. If $v$ solves \eqref{eq:dual-euler}, the corresponding formal function on the sine side of the discrete equation is $f=v/a$, since $af=v$ and $Tv=\lambda f$. Membership of $v$ in $H_0^1(\Om)$ does not imply $v/a\in L^2(\Om)$. The existence statement in Corollary~\ref{cor:attainment} concerns the continuous form domain.

\section{A strict lower bound in dimension three}\label{sec:three}

The coefficient twelve admits a positive supersolution whose residual can be written as a sum of nonnegative terms. Because the dual infimum is attained below $49/4$, pairing this supersolution with the first eigenfunction yields a strict lower bound. The proof uses only exact algebra and the attainment already established.

Throughout this section $N=3$. Retain $S,Q_2,\tau$ from \eqref{eq:q-coordinates} and put
\begin{equation}\label{eq:third-moment}
 Q_3=\sum_{j=1}^3q_j^3,\qquad
 \upsilon=\frac{Q_3}{S^3},\qquad
 G=1+\frac{13}{54}S-\frac{Q_2}{3S},\qquad
 F=S^{-3/2}G.
\end{equation}
Since $Q_2\leq S$, we have $G\geq2/3+13S/54>0$.

\Needspace{14\baselineskip}
The next computation supplies an explicit residual and its sign.

\begin{lemma}\label{lem:three-certificate}
Let $\mathcal J_3$ be defined by \eqref{eq:Jacobi}. Then
\begin{equation}\label{eq:three-residual}
 \bigl[4S(3-\mathcal J_3)-12\bigr]F
                      =S^{-1/2}\mathcal P(S,\tau,\upsilon),
\end{equation}
where
\begin{align}
 \mathcal P(S,\tau,\upsilon)
 &=\frac79(3\tau-1)
 +S\left(\frac{13}{9}-\frac{71}{18}\tau
               +\frac{40}{3}\upsilon-\frac{35}{3}\tau^2\right)
 \label{eq:three-polynomial}\\
 &=\frac79(3\tau-1)(1-S\tau)\notag\\
 &\quad+S\left[
 4\left(\tau-\frac{85}{144}\right)^2
 +\frac{263}{5184}
 +\frac{40}{3}(\upsilon-\tau^2)\right]>0.
 \label{eq:three-sos}
\end{align}
\end{lemma}

\begin{proof}
We compute the residual separately on the three terms of $F$ and then verify its positive decomposition. All coefficients below are rational.

The derivatives of $Q_2$ are $\partial_jQ_2=2q_j$ and $\partial_j^2Q_2=2$. Hence
\[
 \mathcal J_3Q_2
 =\sum_j\left[2q_j(1-q_j)+2q_j\left(\frac32-3q_j\right)\right]
 =5S-8Q_2.
\]
For smooth $f,g$, the product rule for \eqref{eq:Jacobi} is
\[
 \mathcal J_3(fg)=f\mathcal J_3g+g\mathcal J_3f
                 +2\sum_j q_j(1-q_j)\partial_jf\,\partial_jg.
\]
Applying it to $f=Q_2$, $g=S^r$, the last term is
$4r(Q_2-Q_3)S^{r-1}$. Together with \eqref{eq:J-radial}, this yields
\begin{align}
 \mathcal J_3S^r
 &=r\left(r+\frac72\right)S^{r-1}
        -3rS^r-r(r-1)Q_2S^{r-2},\label{eq:J-three-power}\\
 \mathcal J_3(Q_2S^r)
 &=5S^{r+1}-(8+3r)Q_2S^r
   +r\left(r+\frac{15}{2}\right)Q_2S^{r-1}\notag\\
 &\quad-r(r-1)Q_2^2S^{r-2}-4rQ_3S^{r-1}.
 \label{eq:J-three-mixed}
\end{align}

Write $\mathcal A=4S(3-\mathcal J_3)-12$. Substituting $r=-3/2,-1/2$ into \eqref{eq:J-three-power}, and $r=-5/2$ into \eqref{eq:J-three-mixed}, gives the three explicit expressions
\begin{align*}
 \mathcal A(S^{-3/2})
 &=S^{-1/2}(-6+15\tau),\\
 \mathcal A(S^{-1/2})
 &=S^{-1/2}\bigl[-6+S(6+3\tau)\bigr],\\
 \mathcal A(Q_2S^{-5/2})
 &=S^{-1/2}\bigl[-20+38\tau
                   +S(14\tau+35\tau^2-40\upsilon)\bigr].
\end{align*}
Since
$F=S^{-3/2}+(13/54)S^{-1/2}-(1/3)Q_2S^{-5/2}$,
the part of $S^{1/2}\mathcal AF$ independent of $S$ is
\[
 -6-\frac{13}{9}+\frac{20}{3}
       +\left(15-\frac{38}{3}\right)\tau
 =\frac79(3\tau-1).
\]
The coefficient of $S$ is
\[
 \frac{13}{54}(6+3\tau)
 -\frac13(14\tau+35\tau^2-40\upsilon)
 =\frac{13}{9}-\frac{71}{18}\tau
      -\frac{35}{3}\tau^2+\frac{40}{3}\upsilon.
\]
This proves \eqref{eq:three-residual} and \eqref{eq:three-polynomial}.

To obtain the positive decomposition, subtract
$\frac79(3\tau-1)(1-S\tau)$ from this polynomial. The remainder is
\[
 S\left[\frac{40}{3}(\upsilon-\tau^2)
                  +4\tau^2-\frac{85}{18}\tau+\frac{13}{9}\right].
\]
Completing the square gives
\[
 4\tau^2-\frac{85}{18}\tau+\frac{13}{9}
 =4\left(\tau-\frac{85}{144}\right)^2+\frac{263}{5184},
\]
which proves \eqref{eq:three-sos}.

It remains to verify every sign. Put $z_j=q_j/S$, so $z_j>0$, $\sum_jz_j=1$, $\tau=\sum_jz_j^2$ and $\upsilon=\sum_jz_j^3$. Expanding squares gives
\begin{equation}\label{eq:three-moment-signs}
 \begin{split}
 3\tau-1&=\sum_{i<j}(z_i-z_j)^2\geq0,\\
 1-S\tau&=\frac{S-Q_2}{S}
          =\frac{\sum_jq_j(1-q_j)}{S}>0,\\
 \upsilon-\tau^2&=\sum_jz_j(z_j-\tau)^2\geq0.
 \end{split}
\end{equation}
For the last identity, the expansion is
$\sum_jz_j^3-2\tau\sum_jz_j^2+\tau^2\sum_jz_j=\upsilon-\tau^2$.
Every term in \eqref{eq:three-sos} is therefore nonnegative, and the term $263S/5184$ is strictly positive. Thus $\mathcal P>0$ throughout the cube.
\end{proof}

The coefficients in \eqref{eq:third-moment} can be found within a two-parameter family. For
\[
 F_{\alpha,\eta}=S^{-3/2}\left(1+\alpha S+\eta\frac{Q_2}{S}\right),
\]
after factoring out $S^{-1/2}$, the part of the residual independent of $S$ is
\[
 -6-6\alpha-20\eta+(15+38\eta)\tau.
\]
Making it vanish at the symmetric value $\tau=1/3$ gives
$\alpha=-1/6-(11/9)\eta$. The choice $\eta=-1/3$ then yields
$\alpha=13/54$ and the positive decomposition in \eqref{eq:three-sos}.

We next put the supersolution in the form domain. This step makes the final comparison an exact identity between finite integrals.

\begin{proposition}\label{prop:strict-twelve}
The function
\begin{equation}\label{eq:three-supersolution}
 \Phi_*(\theta)=h(\theta)F(q(\theta))
\end{equation}
belongs to $H_0^1(\Omega_3)$, is positive in the interior, and satisfies weakly
\begin{equation}\label{eq:three-positive-remainder}
 \left(-\Delta-\frac{12}{a}\right)\Phi_*=R_*\Phi_*,
 \qquad R_*=\frac{\mathcal P}{4G}>0.
\end{equation}
The function $R_*$ is bounded. Moreover $C_+(3)>12$.
\end{proposition}

\begin{proof}
The lower bound $G\geq2/3+13S/54$ follows from $Q_2\leq S$, since $0<q_j<1$. Thus $F$ and $\Phi_*=hF$ are positive in the interior. By \eqref{eq:conjugation} and Lemma~\ref{lem:three-certificate},
\[
 (3-\mathcal J_3)F-\frac3S F
 =\frac14S^{-3/2}\mathcal P
 =\frac{\mathcal P}{4G}F.
\]
As $a=4S$, multiplication by $h$ proves the pointwise equation \eqref{eq:three-positive-remainder}. The quantities $S,\tau,\upsilon$ are bounded, and $G\geq2/3$, so $R_*$ is bounded. Lemma~\ref{lem:three-certificate} gives $R_*>0$.

We next verify the form-domain assertion. Near the corner, write $r=|\theta|$. Then $S\asymp r^2$, $|\nabla S|=O(r)$, $h=O(r^3)$, and $|\nabla h|=O(r^2)$. Also $Q_2\leq S^2$, so $G=1+O(S)=1+O(r^2)$. Its derivatives are controlled by
\[
 \partial_{q_i}\left(\frac{Q_2}{S}\right)
 =\frac{2q_iS-Q_2}{S^2},\qquad
 \left|\partial_{q_i}\left(\frac{Q_2}{S}\right)\right|\leq3.
\]
Since $\partial_{\theta_i}q_i=O(r)$, this implies $|\nabla G|=O(r)$. Differentiating $hS^{-3/2}G$ now gives
\[
 |\Phi_*|\leq C,\qquad
 |\nabla\Phi_*|\leq C r^{-1}.
\]
The three-dimensional volume element is $r^2\dd r\dd\omega$, and hence
\[
 \int_{\{r<\varepsilon\}}|\nabla\Phi_*|^2\dd\theta\leq C\varepsilon,
 \qquad
 \int_{\{r<\varepsilon\}}\frac{\Phi_*^2}{a}\dd\theta\leq C\varepsilon.
\]
These estimates show that $\Phi_*\in H^1$ and that its weighted mass is finite near the corner.

To prove the zero trace without making a pointwise assertion at the singular corner, take a radial cutoff $\xi_\varepsilon$ which is zero on $r\leq\varepsilon$, one on $r\geq2\varepsilon$, and satisfies $|\nabla\xi_\varepsilon|\leq C/\varepsilon$. The function $\xi_\varepsilon\Phi_*$ is smooth up to the remaining boundary and vanishes there because of $h$, so it belongs to $H_0^1(\Omega_3)$. Moreover,
\begin{align*}
 \|\nabla(\Phi_*-\xi_\varepsilon\Phi_*)\|_2^2
 &\leq 2\int_{\{r<2\varepsilon\}}|\nabla\Phi_*|^2\dd\theta
     +\frac{C}{\varepsilon^2}
        \int_{\{r<2\varepsilon\}}|\Phi_*|^2\dd\theta\\
 &\leq C\varepsilon.
\end{align*}
The $L^2$ difference also tends to zero. Since $H_0^1$ is closed in $H^1$, this proves $\Phi_*\in H_0^1(\Omega_3)$.

The pointwise equation gives its weak form for compactly supported smooth tests. Each term extends continuously to $H_0^1$: for the singular term, Cauchy--Schwarz and \eqref{eq:weight-control} give
\[
 \left|\int_{\Omega_3}\frac{\Phi_*\psi}{a}\dd\theta\right|
 \leq\wform[\Phi_*]^{1/2}\wform[\psi]^{1/2},
\]
and the term involving $R_*$ is bounded by
$\|R_*\|_\infty\|\Phi_*\|_2\|\psi\|_2$. Density therefore proves \eqref{eq:three-positive-remainder} for every form-domain test.

Let $v_3$ be the positive first eigenfunction from Corollary~\ref{cor:attainment}. Both $v_3$ and $\Phi_*$ are admissible as tests in each other's weak equations. These give, respectively,
\begin{align*}
 \int_{\Omega_3}\nabla v_3\cdot\nabla\Phi_*\dd\theta
 &=C_+(3)\int_{\Omega_3}\frac{v_3\Phi_*}{a}\dd\theta,\\
 \int_{\Omega_3}\nabla\Phi_*\cdot\nabla v_3\dd\theta
 &=12\int_{\Omega_3}\frac{\Phi_*v_3}{a}\dd\theta
      +\int_{\Omega_3}R_*\Phi_*v_3\dd\theta.
\end{align*}
Subtracting yields
\begin{equation}\label{eq:strict-pairing}
 \bigl(C_+(3)-12\bigr)
       \int_{\Omega_3}\frac{\Phi_*v_3}{a}\dd\theta
       =\int_{\Omega_3}R_*\Phi_*v_3\dd\theta>0.
\end{equation}
The weighted integral is finite by the preceding Cauchy--Schwarz estimate and strictly positive because both functions are positive. The right-hand side is finite by boundedness of $R_*$ and is strictly positive. Dividing by the positive weighted integral proves $C_+(3)>12$.
\end{proof}

\begin{proof}[Completion of the proof of Theorem~\ref{thm:main}]
The classification \eqref{eq:classification} was proved in Section~\ref{sec:transition}. In dimension three it gives $C_+(3)<A_3=49/4$, using the negative identity \eqref{eq:three-total-defect}. Proposition~\ref{prop:strict-twelve} now gives the other strict inequality, $12<C_+(3)$. This proves \eqref{eq:three-bounds} and completes the theorem.
\end{proof}

Attainment is decisive in the last argument: it turns the positive residual into the strict identity \eqref{eq:strict-pairing}. In dimensions one and two, the positive remainders coexist with equality at the essential threshold because there is no minimizing function in the form domain.

\section{A representation on a product of spheres}\label{sec:geometry}

The coefficient $A_N$ can also be understood through an exact geometric representation. Each Dirichlet coordinate contributes a three-dimensional sphere after conjugation by its sine factor. The resulting product has local dimension $3N$, and the lattice symbol becomes a squared chordal distance.

Let
\[
 M_N=(\mathbb S^3)^N
\]
with the product metric and unit sphere radii. Fix a pole $P_j$ in each factor, and write $P=(P_1,\ldots,P_N)$. For $X=(X_1,\ldots,X_N)\in M_N$, put
\begin{equation}\label{eq:chordal}
 a_M(X)=\sum_{j=1}^N|X_j-P_j|_{\R^4}^2.
\end{equation}
If $\theta_j$ is the polar angle from $P_j$, then
$a_M=\sum_j2(1-\cos\theta_j)=a(\theta)$. The distance in \eqref{eq:chordal} is the ambient chordal distance in the product embedding into $\R^{4N}$.

\begin{theorem}\label{thm:product-spheres}
For every $N\geq1$,
\begin{equation}\label{eq:sphere-quotient}
 \Cp=\inf_{0\ne F\in H^1(M_N)}
 \frac{\displaystyle\int_{M_N}\bigl(|\nabla_{M_N}F|^2+N|F|^2\bigr)\dd V}
      {\displaystyle\int_{M_N}\frac{|F|^2}{a_M}\dd V}.
\end{equation}
The infimum is unchanged when restricted to functions depending only on the polar angles $\theta_1,\ldots,\theta_N$.
\end{theorem}

\begin{proof}
We first identify the quotient on functions invariant under rotations fixing the poles. We then remove this restriction by an averaging operation that preserves the denominator.

\emph{The identity for polar functions.}
On the unit sphere $\mathbb S^3$, polar coordinates give the metric
$\dd\theta^2+\sin^2\theta\,g_{\mathbb S^2}$. Integrating over $\mathbb S^2$ gives the volume factor $4\pi\sin^2\theta\dd\theta$. For a function of $\theta$ alone, its gradient squared is $|F'|^2$ and its Laplacian is
\[
 \frac1{\sin^2\theta}\frac{\mathrm d}{\mathrm d\theta}
       \left(\sin^2\theta\frac{\mathrm dF}{\mathrm d\theta}\right)
 =F''+2\cot\theta F'.
\]
For a smooth function on the product depending only on the polar angles, let $v=hF$. Fix a coordinate $j$ and integrate by parts in that coordinate. The product rule gives
\[
 |\partial_j(hF)|^2
 =h^2|\partial_jF|^2+|F|^2|\partial_jh|^2
       +h\partial_jh\,\partial_j|F|^2.
\]
The boundary term $h\partial_jh|F|^2$ vanishes at both endpoints because $h$ has a sine factor and $F$ is smooth on the compact manifold. The last two terms therefore integrate to
$-\int h\partial_j^2h|F|^2=\int h^2|F|^2$, since $\partial_j^2h=-h$. Summing over $j$ proves
\begin{align}
 \int_{\Om}|\nabla v|^2\dd\theta
 &=\int_{\Om}h^2\bigl(|\nabla F|^2+N|F|^2\bigr)\dd\theta,
 \label{eq:sphere-energy}\\
 \int_{\Om}\frac{|v|^2}{a}\dd\theta
 &=\int_{\Om}\frac{h^2|F|^2}{a}\dd\theta.
 \label{eq:sphere-weight}
\end{align}
The manifold integrals are $(4\pi)^N$ times the corresponding integrals on the right, and $a_M=a$. The common volume factor cancels in the quotient.

\emph{The form domains.}
Let $\mathcal G=\mathrm{SO}(3)^N$ act by rotations fixing the poles. The polar subspace $H^1_{\rm pol}(M_N)$ is its fixed-point subspace and is closed. Smooth polar functions are dense in it: approximate in $H^1$ by smooth functions and average linearly over the compact group. The group acts isometrically on $H^1$, so this averaging is continuous and preserves smoothness.

Equation \eqref{eq:sphere-energy} shows that multiplication by $h$ is an isometry, up to the factor $(4\pi)^N$, from the norm
\[
 \left(\int_{M_N}(|\nabla F|^2+N|F|^2)\dd V\right)^{1/2}
\]
on the polar subspace to the Dirichlet energy norm on its image. Both are complete norms, equivalent to the respective $H^1$ and $H_0^1$ norms. Hence the map extends continuously into $H_0^1(\Om)$ and has closed image. That image contains $C_c^\infty(\Om)$: for such a $v$, the function $F=v/h$ vanishes near every polar endpoint, and thus extends as a smooth polar function on the product of spheres. The image is therefore all of $H_0^1(\Om)$. The weight identity extends by \eqref{eq:weight-control}, or by applying it to differences of smooth approximants. Proposition~\ref{prop:duality} now proves \eqref{eq:sphere-quotient} when the infimum is restricted to polar functions.

\emph{Finiteness of the weighted integral for arbitrary functions.}
The manifold dimension is $d=3N\geq3$. In normal coordinates $x$ at the joint pole $P$, the metric and volume density are uniformly comparable to their Euclidean counterparts on a small ball, and $a_M$ is comparable to $|x|^2$. Choose a smooth cutoff supported in that ball and equal to one on a smaller ball. Apply the Euclidean Hardy inequality with coefficient $(d-2)^2/4$ to the cutoff times a smooth $F$. This inequality follows from \eqref{eq:radial-hardy-identity} in dimension $d$, after dropping the nonnegative angular energy; a cutoff at the origin is removed by density since $d\geq3$. The derivatives of the cutoff are supported away from $P$, and the weight is bounded outside the smaller ball. These observations give
\begin{equation}\label{eq:manifold-weight-control}
 \int_{M_N}\frac{|F|^2}{a_M}\dd V
 \leq C\int_{M_N}(|\nabla F|^2+|F|^2)\dd V.
\end{equation}
Smooth density extends this estimate to $H^1(M_N)$. In particular, all quotients in the theorem are well defined.

\emph{Removal of the polar restriction.}
Normalize Haar measure on $\mathcal G$ to have mass one, and set
\[
 \overline F(X)=\left(\int_{\mathcal G}|F(gX)|^2\dd g\right)^{1/2}.
\]
The group orbits are exactly the sets with prescribed polar angles, so $\overline F$ is polar. Invariance of volume and of $a_M$, followed by Fubini's theorem, gives
\begin{align}
 \int_{M_N}|\overline F|^2\dd V&=\int_{M_N}|F|^2\dd V,
 \label{eq:average-mass}\\
 \int_{M_N}\frac{|\overline F|^2}{a_M}\dd V
 &=\int_{M_N}\frac{|F|^2}{a_M}\dd V.
 \label{eq:average-weight}
\end{align}
For smooth $F$, regularize the square root by
$F_\varepsilon=(\varepsilon^2+\int_{\mathcal G}|F(gX)|^2\dd g)^{1/2}$. Differentiation and Cauchy--Schwarz imply
\[
 |\nabla F_\varepsilon|^2
 \leq\frac{\int_{\mathcal G}|F(gX)|^2\dd g}
              {\varepsilon^2+\int_{\mathcal G}|F(gX)|^2\dd g}
       \int_{\mathcal G}|\nabla(F\circ g)(X)|^2\dd g
 \leq\int_{\mathcal G}|\nabla(F\circ g)(X)|^2\dd g.
\]
Since the action is isometric, integration gives
$\int|\nabla F_\varepsilon|^2\leq\int|\nabla F|^2$. As $\varepsilon\downarrow0$, $F_\varepsilon\to\overline F$ in $L^2$; weak compactness in $H^1$ and lower semicontinuity prove
\begin{equation}\label{eq:average-energy}
 \int_{M_N}|\nabla\overline F|^2\dd V
 \leq\int_{M_N}|\nabla F|^2\dd V.
\end{equation}

For general $F\in H^1$, take smooth $F_m\to F$ in $H^1$. The triangle inequality in $L^2(\mathcal G)$ gives pointwise
\[
 |\overline F_m-\overline F|^2
 \leq\int_{\mathcal G}|F_m(gX)-F(gX)|^2\dd g.
\]
Thus the averaged functions converge in $L^2$ and also in the weighted $L^2$ norm, by invariance and \eqref{eq:manifold-weight-control}. Their gradients are bounded by \eqref{eq:average-energy}. Passing to a weakly convergent subsequence proves that $\overline F\in H^1_{\rm pol}$ and that \eqref{eq:average-energy} holds for $F$. Equations \eqref{eq:average-mass}--\eqref{eq:average-energy} preserve both mass terms and decrease the gradient term. If $F\ne0$, then $\overline F\ne0$ and its quotient is no larger. The unrestricted infimum is consequently at least the polar infimum; the opposite inequality follows by inclusion. This proves the theorem.
\end{proof}

The geometry reproduces the identities used in the analytic proof. Write
$a_j=2(1-\cos\theta_j)$, so that $a_M=\sum_ja_j$. Then
\begin{equation}\label{eq:sphere-geometric-identities}
 \Delta_{M_N}a_M=6N-3a_M,
 \qquad
 |\nabla_{M_N}a_M|^2=4a_M-\sum_{j=1}^Na_j^2.
\end{equation}
The first identity follows from $\Delta_{\mathbb S^3}\cos\theta=-3\cos\theta$, and the second from $|\nabla a_j|^2=4a_j-a_j^2$. The term $\sum_ja_j^2$ retains the product geometry; in the coordinates of Section~\ref{sec:transition} it is $16Q_2$.

At the joint pole, the manifold has dimension $3N$ and $a_M$ agrees to leading order with the squared geodesic distance. Its local inverse-square threshold is therefore $(3N-2)^2/4$. The bounded mass term $N$ in \eqref{eq:sphere-quotient} does not change that local coefficient. The global geometry, recorded by the residual in \eqref{eq:critical-residual}, determines the strict decrease in dimensions at least three.

\Needspace{11\baselineskip}
The same correspondence can be seen at the level of orthogonal bases. The polar functions
\[
 \chi_n(\theta)=\frac{\sin(n\theta)}{\sin\theta},\qquad n\geq1,
\]
are orthonormal in $L^2((0,\pi),(2/\pi)\sin^2\theta\dd\theta)$ and satisfy
\[
 (-\Delta_{\mathbb S^3}+1)\chi_n=n^2\chi_n,
 \qquad 2\cos\theta\,\chi_n=\chi_{n+1}+\chi_{n-1},\qquad\chi_0=0.
\]
Indeed the polar Laplacian satisfies
\[
 (-\Delta_{\mathbb S^3}+1)F
 =-\frac1{\sin\theta}\frac{\mathrm d^2}{\mathrm d\theta^2}
       (\sin\theta\,F).
\]
Taking $F=\chi_n$ proves the eigenvalue equation, while the sine addition formula proves the recurrence. Orthonormality reduces to that of the sine basis after cancelling the two factors of $\sin\theta$. Tensorization identifies multiplication by $|n|^2$ with $-\Delta_{M_N}+N$ on the polar subspace and identifies $L_+$ with multiplication by $a_M$. This gives another form of the operator correspondence behind Proposition~\ref{prop:duality}.

\section{Further comparison problems}\label{sec:questions}

The preceding results distinguish three tasks: determining whether a continuous sharp constant survives discretization, locating the transition when it fails, and evaluating the new discrete optimum. The present paper completes the first two tasks on the positive lattice. An explicit evaluation of $C_+(3)$, and more generally of $C_+(N)$ for $N\geq3$, remains open. The attained weighted eigenvalue problem uniquely selects the first eigenfunction, but the residual identity by itself does not evaluate its eigenvalue.

\subsection{Varying the number of Dirichlet faces}

A natural family interpolates between the full lattice and the positive lattice:
\[
 \mathbb Z^{N-k}\times\N^k,\qquad 0\leq k\leq N.
\]
For $N\geq3$, the corresponding continuous region
$\R^{N-k}\times(0,\infty)^k$ has sharp inverse-square coefficient
\[
 A_{N,k}=\frac{(N+2k-2)^2}{4};
\]
this is the family treated by Su and Yang \cite{SuYang}. At $k=0$, the discrete problem is taken with the zero value at the origin, as in \cite{LizamaEarly,Alpay,LizamaFull}; at $k>0$, Dirichlet conditions are imposed on the coordinate faces. The two endpoints have different transition dimensions. It is therefore natural to determine the equality region in the pair $(N,k)$ and to ask whether tangential identities can characterize its boundary. Such a result would separate the effects of ambient dimension and the number of boundary constraints.

\subsection{Varying the weight and the energy}

The weighted half-line results of \cite{GuptaWeighted,Barki} show that weight parameters can govern preservation of the continuous coefficient even in one dimension. General weighted discrete inequalities also admit sharp remainder terms and quantitative stability estimates, as shown by Yessirkegenov and Zhangirbayev \cite{YessirkegenovZhangirbayev}. In arbitrary dimension, the power-weighted existence results of Dyda \cite{Dyda} identify families for which the sharp comparison can be investigated. The Bessel-pair theory in \cite{GhoussoubMoradifam} supplies a larger collection of continuous weighted inequalities whose optimal coefficients can be characterized by an ordinary differential equation. For each such comparison, the lattice energy, the sampled or otherwise specified discrete weight, and the boundary convention must first be fixed. The resulting question is whether the exact continuous coefficient persists for that discretization and, if it does not, what global quantity detects the loss.

The trigonometric weight $a^{-1}$ in the present paper has a particularly useful feature: its continuous variational problem is exactly equivalent to the discrete one. Analogous exact correspondences would allow continuous tools to determine discrete coefficients. The product-of-spheres representation shows at the same time that local dimension alone does not settle the global problem. A general theory would need to account for the full geometry or symbol, in addition to its quadratic behaviour near the singular point.

\subsection{Varying symmetry and order}

The antisymmetric inequalities in \cite{GuptaAntisymmetric} and the higher-order inequalities in \cite{HuangYe,Gupta} give further families with continuous benchmarks. The large-dimensional asymptotics in \cite{GuptaAsymptotic,HuangYeAsymptotic} establish a precise regime of separation for higher orders. One can ask for the exact equality ranges and transition parameters in dimension, order or symmetry class. The associated admissible spaces and singular sets must be retained in the comparison. The present identity suggests a concrete strategy: identify a critical continuous profile, calculate its defect for an exact transformed problem, and seek a global identity whose sign decides whether the benchmark remains optimal. Establishing such identities beyond the positive lattice is a substantive open step.

\section*{Data availability}
No datasets were generated or analysed during this study. 

\section*{AI Declaration}

During the preparation of this manuscript, the author used OpenAI's ChatGPT-5.6 to assist with language editing, improve the clarity and organization of the exposition, and provide supportive checks for certain mathematical calculations and arguments. All mathematical statements, proofs, derivations, and conclusions are solely the responsibility of the author, who has rigorously verified them.


\begin{thebibliography}{99}

\bibitem{AlpayEarly}
N. Alpay,
\emph{The sharp discrete Hardy inequality on $\mathbb Z^3$},
preprint (2026), \href{https://arxiv.org/abs/2608.25262}{arXiv:2608.25262}.

\bibitem{Alpay}
N. Alpay,
\emph{Sharp discrete Hardy constants in low dimensions},
preprint (2026), \href{https://arxiv.org/abs/2610.05421}{arXiv:2610.05421}.

\bibitem{Barki}
A. Barki,
\emph{Sharp weighted discrete $p$-Hardy inequality and stability},
preprint (2024), revised 2026,
\href{https://arxiv.org/abs/2501.00299}{arXiv:2501.00299}.

\bibitem{Davies}
E.~B. Davies,
\emph{A review of Hardy inequalities},
in \emph{The Maz'ya Anniversary Collection}, vol.~2,
Operator Theory: Advances and Applications, vol.~110,
Birkh\"auser, Basel, 1999, pp.~55--67.
\href{https://doi.org/10.1007/978-3-0348-8672-7_5}{doi:10.1007/978-3-0348-8672-7\_5}.

\bibitem{Dyda}
B. Dyda,
\emph{Fractional and non-fractional Hardy inequality on a lattice $\mathbb Z^d$},
J. Differential Equations \textbf{485} (2026), article~114749.
\href{https://doi.org/10.1016/j.jde.2026.114749}{doi:10.1016/j.jde.2026.114749}.

\bibitem{FrankSeiringer}
R.~L. Frank and R. Seiringer,
\emph{Non-linear ground state representations and sharp Hardy inequalities},
J. Funct. Anal. \textbf{255} (2008), no.~12, 3407--3430.
\href{https://doi.org/10.1016/j.jfa.2008.05.015}{doi:10.1016/j.jfa.2008.05.015}.

\bibitem{GhoussoubMoradifam}
N. Ghoussoub and A. Moradifam,
\emph{Bessel pairs and optimal Hardy and Hardy--Rellich inequalities},
Math. Ann. \textbf{349} (2011), no.~1, 1--57.
\href{https://doi.org/10.1007/s00208-010-0510-x}{doi:10.1007/s00208-010-0510-x}.

\bibitem{GuptaWeighted}
S. Gupta,
\emph{Discrete weighted Hardy inequality in 1-D},
J. Math. Anal. Appl. \textbf{514} (2022), no.~2, article~126345.
\href{https://doi.org/10.1016/j.jmaa.2022.126345}{doi:10.1016/j.jmaa.2022.126345}.

\bibitem{Gupta}
S. Gupta,
\emph{Hardy and Rellich inequality on lattices},
Calc. Var. Partial Differential Equations \textbf{62} (2023), article~81.
\href{https://doi.org/10.1007/s00526-022-02407-0}{doi:10.1007/s00526-022-02407-0}.

\bibitem{GuptaAntisymmetric}
S. Gupta,
\emph{Hardy inequalities for antisymmetric functions},
Nonlinear Anal. \textbf{248} (2024), article~113619.
\href{https://doi.org/10.1016/j.na.2024.113619}{doi:10.1016/j.na.2024.113619}.

\bibitem{GuptaAsymptotic}
S. Gupta,
\emph{Sharp asymptotics for higher-order Hardy constants on lattices},
preprint (2026), \href{https://arxiv.org/abs/2607.15181}{arXiv:2607.15181}.

\bibitem{HLP}
G.~H. Hardy, J.~E. Littlewood and G. P\'olya,
\emph{Inequalities}, 2nd ed., Cambridge University Press, Cambridge, 1952.

\bibitem{HuangYe}
X. Huang and D. Ye,
\emph{One-dimensional sharp discrete Hardy--Rellich inequalities},
J. Lond. Math. Soc. \textbf{109} (2024), no.~1, e12838.
\href{https://doi.org/10.1112/jlms.12838}{doi:10.1112/jlms.12838}.

\bibitem{HuangYeAsymptotic}
X. Huang and D. Ye,
\emph{Asymptotically sharp Hardy--Rellich inequalities on lattices},
preprint (2026), \href{https://arxiv.org/abs/2607.23057}{arXiv:2607.23057}.

\bibitem{KapitanskiLaptev}
L. Kapitanski and A. Laptev,
\emph{On continuous and discrete Hardy inequalities},
J. Spectr. Theory \textbf{6} (2016), no.~4, 837--858.
\href{https://doi.org/10.4171/JST/144}{doi:10.4171/JST/144}.

\bibitem{KellerLemm}
M. Keller and M. Lemm,
\emph{Optimal Hardy weights on the Euclidean lattice},
Trans. Amer. Math. Soc. \textbf{376} (2023), no.~9, 6033--6062.
\href{https://doi.org/10.1090/tran/8939}{doi:10.1090/tran/8939}.

\bibitem{KPP}
M. Keller, Y. Pinchover and F. Pogorzelski,
\emph{Optimal Hardy inequalities for Schr\"odinger operators on graphs},
Comm. Math. Phys. \textbf{358} (2018), no.~2, 767--790.
\href{https://doi.org/10.1007/s00220-018-3107-y}{doi:10.1007/s00220-018-3107-y}.

\bibitem{KPPRellich}
M. Keller, Y. Pinchover and F. Pogorzelski,
\emph{From Hardy to Rellich inequalities on graphs},
Proc. Lond. Math. Soc. \textbf{122} (2021), 458--477.
\href{https://doi.org/10.1112/plms.12376}{doi:10.1112/plms.12376}.

\bibitem{LizamaEarly}
C. Lizama,
\emph{Sharp discrete Hardy constants in dimensions three and four and strict upper bounds from dimension nine},
preprint (2026), \href{https://arxiv.org/abs/2608.30831}{arXiv:2608.30831}.

\bibitem{LizamaFull}
C. Lizama,
\emph{Dimensional transition and spectral instability in lattice Hardy inequalities},
preprint (2026), \href{https://arxiv.org/abs/2610.14950}{arXiv:2610.14950}.

\bibitem{ReedSimon}
M. Reed and B. Simon,
\emph{Methods of Modern Mathematical Physics. IV: Analysis of Operators},
Academic Press, New York, 1978.

\bibitem{SuYang}
D. Su and Q.-H. Yang,
\emph{On the best constants of Hardy inequality in
$\R^{n-k}\times(\R_+)^k$ and related improvements},
J. Math. Anal. Appl. \textbf{389} (2012), no.~1, 48--53.
\href{https://doi.org/10.1016/j.jmaa.2011.11.033}{doi:10.1016/j.jmaa.2011.11.033}.

\bibitem{YessirkegenovZhangirbayev}
N. Yessirkegenov and A. Zhangirbayev,
\emph{Sharp forms and quantitative stability for general weighted discrete $p$-Hardy inequalities},
preprint (2026), \href{https://arxiv.org/abs/2604.02229}{arXiv:2604.02229}.

\end{thebibliography}
\end{document}